\documentclass[amscd,amssymb,11pt]{amsart}

\usepackage{amssymb}
\usepackage[all]{xy}

\newtheorem{thm}{Theorem}[section]
\newtheorem{lem}[thm]{Lemma}
\newtheorem{cor}[thm]{Corollary}
\newtheorem{prop}[thm]{Proposition}

\theoremstyle{definition}

\newtheorem{defn}[thm]{Definition}

\newtheorem{quest}[thm]{Question}

\newtheorem{ex}[thm]{Example}

\theoremstyle{remark}

\newtheorem{rem}[thm]{Remark}

\numberwithin{equation}{section}

\newcommand{\thmref}[1]{Theorem~\ref{#1}}
\newcommand{\corref}[1]{Corollary~\ref{#1}}
\newcommand{\secref}[1]{\S\ref{#1}}

\newcommand{\propref}[1]{Proposition~\ref{#1}}
\newcommand{\lemref}[1]{Lemma~\ref{#1}}

\newcommand{\hocolim}{\operatorname*{hocolim}}
\newcommand{\holim}{\operatorname*{holim}}
\newcommand{\colim}{\operatorname*{colim}}

\newcommand{\Map}{\operatorname{Map}}
\newcommand{\MapS}{\operatorname{Map_{\mathcal S}}}
\newcommand{\MapT}{\operatorname{Map_{\mathcal T}}}

\newcommand{\A}{{\mathcal  A}}

\newcommand{\C}{{\mathcal  C}}
\newcommand{\D}{{\mathcal  D}}
\newcommand{\Sp}{{\mathcal  S}}

\newcommand{\T}{{\mathcal  T}}

\newcommand{\Z}{{\mathbb  Z}}

\newcommand{\Sinfty}{\Sigma^{\infty}}
\newcommand{\Oinfty}{\Omega^{\infty}}

\newcommand{\sm}{\wedge}

\newcommand{\ra}{\rightarrow}
\newcommand{\xra}{\xrightarrow}
\newcommand{\la}{\leftarrow}
\newcommand{\xla}{\xleftarrow}

\begin{document}

\title[Telescopic functors]{A guide to telescopic functors}

\author[Kuhn]{Nicholas J.~Kuhn}

\address{Department of Mathematics \\ University of Virginia \\ Charlottesville, VA 22903}

\email{njk4x@virginia.edu}

\thanks{This research was partially supported by a grant from the National Science Foundation}

\date{February 2, 2008.}


\begin{abstract}  In the mid 1980's, Pete Bousfield and I constructed certain $p$--local `telescopic' functors $\Phi_n$ from spaces to spectra, for each prime $p$, and each $n \geq 1$.
They are constructed using the full strength of the Nilpotence and Periodicity Theorems of Devanitz--Hopkins--Smith, and have some striking properties that relate the chromatic approach to homotopy theory to infinite loopspace theory.

Recently there have been a variety of new uses of these functors, suggesting that they have a central role to play in calculations of periodic phenomena.  Here I offer a guide to their construction, characterization, application, and computation.

\end{abstract}

\maketitle

\section{Introduction} \label{introduction}

Let $\T$ and $\Sp$ be the categories of based spaces and spectra, localized at a fixed prime $p$, and $\Sinfty: \T \ra \Sp$ and $\Oinfty: \T \ra \Sp$ the usual adjoint pair. For $n \geq 1$, in \cite{k1}, the author constructed functors between the homotopy categories
$$ \Phi^K_n: ho(\T) \ra ho(\Sp)$$
such that
\begin{equation*}\Phi^K_n(\Oinfty X) \simeq L_{K(n)}X,
\end{equation*}
 where $L_{K(n)}$ denotes Bousfield localization with respect to the Morava $K$--theory $K(n)$.
Thus the $K(n)$--localization of a spectrum depends only on its zero space.

This mid 1980's result was modeled on the $n=1$ version that had been newly established by Pete Bousfield in \cite{bousfield1}, and heavily used the newly proved Nilpotence and Periodicity Theorems of Ethan Devanitz, Mike Hopkins, and Jeff Smith \cite{dhs,hs}.

Bousfield's main application was to proving uniqueness results about infinite loopspaces; for example, he gives a `conceptual' proof of the Adams--Priddy theorem \cite{adamspriddy} that $BSO_{(p)}$ admits a unique infinite loopspace structure up to homotopy equivalence.  My main application was to note that the evaluation map $\epsilon: \Sinfty \Oinfty X \ra X$ has a section after applying $L_{K(n)}$, and thus after applying other functors like $K(n)_*$.

The functors $\Phi^K_n$ as described above allow for two important refinements.

Firstly, let $T(n)$ be the mapping telescope of any $v_n$--self map of a finite CW spectrum of type $n$.  It is a well known application of the Periodicity Theorem that the associated localization functor $L_{T(n)}$ is independent of choices, and it is evident that $K(n)$--local objects are $T(n)$--local\footnote{The Telescope Conjecture, open for $n \geq 2$, asserts that the converse is also true, so that $L_{T(n)} = L_{K(n)}$.}. This suggests that $\Phi^K_n$ might refine to a functor
$$ \Phi^T_n: ho(\T) \ra ho(\Sp)$$
with $L_{K(n)}\circ \Phi^T_n = \Phi^K_n$, such that
$$\Phi^T_n(\Oinfty X) \simeq L_{T(n)}X.$$
Indeed, a careful reading of the arguments in \cite{k1} shows that this is the case.  One application of this refined functor is that there is a natural isomorphism of graded homotopy groups
$$ [B, \Phi^T_n(Z)]_* \simeq v^{-1}\pi_*(Z;B),$$
for all spaces $Z$, where $v: \Sigma^d B \ra B$ is any unstable $v_n$ self map.  Thus the spectrum $\Phi^T_n(Z)$ determines `periodic unstable homotopy'.

Secondly, what one {\em really} wishes to have is a functor on the level of model categories,
$$ \Phi_n: \T \ra \Sp,$$
inducing $\Phi^T_n$ on associated the homotopy categories.  Once again, inspection of the papers \cite{bousfield1, k1} suggests that this should be possible. However, it was not until Bousfield revisited these constructions in his 2001 paper \cite{bousfield3} that this was carefully worked out. One new consequence that emerged was Bousfield's beautiful theorem that every spectrum is naturally $T(n)$--equivalent to a suspension spectrum.

Along with Bousfield's new application, there has been recent use of $\Phi_n$ by the author \cite{k2,k3} and C. Rezk \cite{rezk}, and new methods for computation available using the work of Arone and Mahowald in \cite{aronemahowald}. All of this suggests that the functors $\Phi_n$ have a fundamental role in the study of homotopy, both stable and unstable, as chromatically organized.

Bousfield's detailed paper \cite{bousfield3} is not an easy read: the partial ordering on that paper's set of lemmas, propositions, and theorems induced by the logical flow of the proof structure is poorly correlated with the numerical total ordering.  One could make a similar statement about \cite{bousfieldJAMS}, on which \cite{bousfield3} relies in essential ways.

By constrast, my paper \cite{k1} offers a quite direct approach to the construction of the $\Phi^K_n$, while being admittedly short on detail. If one fills in details, and adds refinement as described above, it emerges that Bousfield and I have slightly different constructions.  It turns out that both flavors satisfy basic characterizing properties, and thus they are naturally equivalent.

Motivated by all of this, here we offer a guide to the $\Phi_n$. This includes
\begin{itemize}
\item a listing of basic properties, and characterization of the functors by some of these,
\item a step by step discussion of their construction, including model category issues that arise,
\item Bousfield's `left adjoint' $\Theta_n: \Sp \ra \T$ and its basic application,
\item a discussion of the uniqueness of the section to $L_{T(n)}(\epsilon)$, and
\item a discussion of calculations of $\Phi_n(Z)$ for various spaces $Z$ including spheres.
\end{itemize}

Though most of the results surveyed appear in the literature, a few haven't.  Among those that have, I have tweaked the order in which they are `revealed'.  For example, and most significantly, our \thmref{v Tn thm}, which describes important properties of the functor $\Phi_v$ (see just below) when $v$ is a $v_n$--self map, is proved in a direct manner here, enroute to proving our main theorem \thmref{main thm}, which lists important properties of $\Phi_n$. By constrast, in \cite{bousfield3}, these properties of $\Phi_v$ first occur as consequences of the properties of $\Phi_n$.  We hope readers appreciate such unknotting of the logic.

We end this introduction by stating a new characterization of the $\Phi_n$.

We need to briefly describe the basic construction on which the functors $\Phi_n$ are based.
A self map of a space $v: \Sigma^d B \ra B$ with $d>0$, induces a natural transformation
$ v(Z): \Map_{\T}(B,Z) \ra \Omega^d\Map_{\T}(B,Z)$
for all spaces $Z \in \T$. The map $v(Z)$ then can be used to define a periodic spectrum $\Phi_v(Z)$ of period $d$, such that
$$ \pi_*(\Phi_v(Z)) \simeq \colim \{[B,Z]_* \xra{v} [B,Z]_{*+d} \xra{v} \dots\} = v^{-1}\pi_*(Z;B).$$

\begin{thm} \label{main thm} For each $n\geq 1$, there is a continuous functor $\Phi_n: \T \ra \Sp$ satisfying the following properties. \\

\noindent{\bf (1)} \ $\Phi_n(Z)$ is $T(n)$--local, for all spaces $Z$. \\

\noindent{\bf (2)} \ There is a weak equivalence of spectra $\Map_{\Sp}(B,\Phi_n(Z)) \simeq \Phi_v(Z)$,
for all unstable $v_n$ self maps $v:\Sigma^d B \ra B$, natural in both $Z$ and $v$. \\

\noindent{\bf (3)} \ There is a natural weak equivalence $\Phi_n(\Oinfty X) \simeq L_{T(n)}X$, for all $\Omega$--spectra $X$. \\

Furthermore, properties {\bf (1)} and {\bf (2)} characterize $\Phi_n$, up to weak equivalence of functors.
\end{thm}

The rest of the paper is organized as follows.  Background material, about both the model category of spectra and periodic homotopy, is given in \secref{background}.  In \secref{phi_v section}, we present the basic theory of the telescopic functor $\Phi_v$ associated to a self map $v:\Sigma^d B \ra B$, and, in \secref{phi v:part 2}, we study $\Phi_v$ when $v$ is additionally a $v_n$--self map.  In \secref{phi_n section}, we define $\Phi_n$, and the proof of \thmref{main thm} follows quickly from the previous results. The adjoint $\Theta_n$ is defined in \secref{theta_n section}, and using it, we prove Bousfield's theorem that spectra are $T(n)$--equivalent to suspension spectra. A short discussion about the section to the $T(n)$--localized evaluation map is given in \secref{eta_n section}. Finally, in \secref{computations}, we offer a brief guide to known computations of $\Phi_n(Z)$ and periodic homotopy groups.

An outline of this material was presented in a talk at the special session on homotopy theory at the A.M.S.~ meeting held in Newark, DE in April, 2005.  I would like to offer my congratulations to Martin Bendersky, Don Davis, Doug Ravenel, and Steve Wilson -- the 60th birthday boys of that session and the March, 2007 conference at Johns Hopkins University -- and thank them all for setting fine examples of grace and enthusiasm to we algebraic topologists who have followed.

\section{Background} \label{background}
\subsection{Categories of spaces and spectra} \label{categories}

These days, it seems prudent to be precise about our categories of `spaces' and `spectra', and needed model category structures.

We will let $\T$ denote the category of based compactly generated topological spaces, though one could as easily work instead with the category of based simplical sets, as Bousfield always does.

Regarding spectra, we would like a single map of the form $C \ra \Omega^d C$ to specify a spectrum.  This suggests using the `plain vanilla' category of (pre)spectra ($\mathcal N$--spectra in \cite{mmss}).

An object $X$ in the category $\Sp$ is a sequence of spaces in $\T$, $X_0, X_1, \dots$, together with a sequence of based maps $\sigma^X_n: \Sigma X_n \ra X_{n+1}$, or, equivalently, $\tilde{\sigma}^X_n: X_n \ra \Omega X_{n+1}$, for $n \geq 0$.  A morphism $f:X \ra Y$ in $\Sp$ is then a sequence of based maps $f_n: X_n \ra Y_n$ such that the diagram
\begin{equation*}
\xymatrix{
 \Sigma X_n\ar[d]^{\sigma^X_n} \ar[r]^{\Sigma f_n} & \Sigma Y_n \ar[d]^{\sigma^Y_n}  \\
X_{n+1} \ar[r]^{f_{n+1}} & Y_{n+1} }
\end{equation*}
commutes for all $n$.

The category $\Sp$ is a topological category; in particular $\Map_{\Sp}(X,Y)$ is an object in $\T$. It is also tensored and cotensored over $\T$, with $A \sm X$ and $\MapS(A,X)$ denoting the tensor and cotensor product of $A \in \T$ with $X \in \Sp$. (See \cite[p.447]{mmss} for more detail.)  We let $\Sigma^d X$ and $\Omega^d X$ denote $S^d \sm X$ and $\MapS(S^d, X)$, as is usual.

The adjoint pair $\Sinfty: \T  \rightleftarrows \Sp: \Oinfty$ is defined by letting
$(\Sinfty A)_n = \Sigma^n A$ and $\Oinfty X = X_0$.  For $d\geq 0$, we let $s^d: \Sp \ra \Sp$ be the $d$--fold shift functor with $(s^d X)_n = X_{n+d}$.  This admits a left adjoint $s^{-d}:\Sp \ra \Sp$ with
\begin{equation*}
(s^{-d}X)_n =
\begin{cases}
X_{n-d} & \text{for } n \geq d \\ * & \text{for } 0\leq n \leq d.
\end{cases}
\end{equation*}
Composing these adjoints, we see that $s^{-d} \circ \Sinfty: \T \ra \Sp$ is left adjoint to the functor sending a spectrum $X$ to its $d^{th}$ space $X_d$.

\subsection{Model category structures} \label{model categories}

We describe model category structures on $\T$ and $\Sp$.

Our category $\T$ is endowed with the `usual' model category structure (see, e.g. \cite{ds}): the weak equivalences are the weak homotopy equivalences, the fibrations are the Serre fibrations, and the cofibrations are retracts of generalized CW inclusions. (We recall that $f: A \ra B$ in $\T$ is a weak homotopy equivalence if, for each point $a \in A$, $f_*: \pi_*(A,a) \ra \pi_*(B,f(a))$ is a bijection, and is a Serre fibration if it has the right lifting property with respect to the maps $D^n \hookrightarrow D^n \sm I_+$.)

Starting from this model category structure on $\T$, $\Sp$ is given its stable model category structure `in the usual way', as in \cite{schwede, hovey2, mmss}, all of which follow the lead of \cite{bf}.

Firstly, $\Sp$ has its `level' model structure\footnote{The name `level' model structure is used in \cite[\S 6]{mmss}.  Schwede \cite{schwede} refers to this as `strict', and Hovey \cite{hovey2} uses `projective'.} in which the weak equivalences and fibrations are the maps $f: X \ra Y$ such that the levelwise maps $f_n: X_n \ra Y_n$ are weak equivalences and fibrations in $\T$ for all $n$.  It is then easy to check that $f$ is a cofibration exactly when the induced maps $X_0 \ra Y_0$ and $X_{n+1} \cup_{\Sigma X_n} \Sigma Y_n \ra Y_{n+1}$ are cofibrations in $\T$.    When needed, we will write $\Sp_l$ for the category of spectra with the level model structure.

Now we need to change the model structure to build in stability. Hovey's general method \cite{hovey2} yields the following in our situation.  We call a spectrum $X$ an {\em $\Omega$--spectrum} if $\tilde{\sigma}_n: X_n \ra \Omega X_{n+1}$ is a weak equivalence in $\T$ for all $n$. Using our adjunctions, this rewrites as the statement that $Map_{\Sp}(i_n,X)$ is a weak equivalence in $\T$ for all $n$, where $i_n: s^{-(n+1)}\Sinfty S^{1} \ra s^{-n}\Sinfty S^0$ is the canonical map in $\Sp$.  Let $Q: \Sp_l \ra \Sp_l$ denote Bousfield localization (as in \cite{hirschhorn}) with respect to the set of map $\{ i_n, n \geq 0\}$.  Then \cite[Thm.2.2]{hovey2} says that there is a stable model structure on $\Sp$ with cofibrations equal to level cofibrations, with weak equivalences the maps $f: X \ra Y$ such that $Qf: QX \ra QY$ is a level equivalence, and with fibrant objects the level fibrant $\Omega$--spectra.

There are two alternative characterizations of the stable equivalences. It is formal to see that $Qf: QX \ra QY$ is a weak level equivalence if and only if $f^*:[Y,Z]_l \ra [X,Z]_l$ is a bijection for all $\Omega$--spectra $Z$, where $[Y,Z]_l$ denotes homotopy classes computed using the level model structure.  True, but {\em not} formal, is the fact that weak equivalences are precisely maps of spectra inducing isomorphisms on $\displaystyle \pi_*(X) = \colim_{n} \pi_{*+n}(X_n)$: see \cite[Proposition 8.7]{mmss} for a clear discussion of this point.

It is easy to check that $\Sp$ is a topological model category in the sense of \cite[Definition  4.2]{ekmm}, so that, for all spectra $X$ and $Y$, $$[X,Y] = \pi_0(\MapS(X^{cof},Y^{fib})),$$ where $X^{cof}$ and $Y^{fib}$ are respectively cofibrant and fibrant replacements for $X$ and $Y$. (Compare with \cite[Proposition 3.10]{goerss jardine} for a nice presentation in the simplicial setting.)

Handy observations include that the evident natural maps $\Sigma^d X \ra s^d X$ and $s^{-d}X \ra \Omega^d X$ are stable equivalences.  Also useful in calculation is that, if $X^{fib}$ is a fibrant replacement for a spectrum $X$, then each of the evident maps
$$ \Oinfty X^{fib} \ra \hocolim_n \Omega^n X_n^{fib} \la \hocolim_n \Omega^n X_n$$
is a weak equivalence of spaces.

In one proof of ours - the proof of \thmref{ho thm} - we make use of function spectra in the homotopy category of spectra\footnote{Bousfield similarly needs this: see \cite[Thm. 11.9]{bousfield3}.}: these exist in $ho(\Sp)$ using well known `naive' constructions in $\Sp$.  To summarize our overuse of the notation $\MapS(X,Y)$:
\begin{itemize}
\item $\MapS(X,Y)$ is in $\T$ for $X,Y \in \Sp$,
\item $\MapS(X,Y)$ is in $\Sp$ for $X \in \T$ and $Y \in \Sp$, and
\item $\MapS(X,Y)$ is in $ho(\Sp)$ for $X,Y \in ho(\Sp)$
\end{itemize}
We trust our meaning will be clear in context.

We end this subsection with a useful lemma and corollary.

\begin{lem}[Compare with {\cite[Lemma 3.3]{k1}}] \label{hocolim lemma} Given a diagram of spectra
$ X(0) \ra X(1) \ra X(2) \ra \dots$
and an increasing sequence of integers
$ 0 \leq d_0 < d_1 < d_2 < \dots$,
the natural diagram of spectra
\begin{equation*}
\xymatrix @-1.1pc{
s^{-d_1}\Sinfty \Sigma^{d_1-d_0}X(0)_{d_0} \ar[d]_{\wr} \ar[dr] & s^{-d_2}\Sinfty \Sigma^{d_2-d_1}X(1)_{d_1} \ar[d]_{\wr} \ar[dr] & s^{-d_3}\Sinfty \Sigma^{d_3-d_2}X(2)_{d_2} \ar[d]_{\wr} \ar@{-->}[dr] &  \\
s^{-d_0}\Sinfty X(0)_{d_0} \ar[d]& s^{-d_1}\Sinfty X(1)_{d_1} \ar[d]& s^{-d_2}\Sinfty X(2)_{d_2} \ar[d]&  \\
X(0) \ar[r] & X(1) \ar[r] & X(2) \ar@{-->}[r] & }
\end{equation*}
induces a weak equivalence between the homotopy colimit of the top zig-zag and the homotopy colimit of the bottom.
\end{lem}
\begin{proof}[Sketch proof] One checks that the map induces an isomorphism on $\pi_*$.  Alternatively, one can check that the map induces an isomorphism on $[\_\_,Y]$ for all $Y \in \Sp$.
\end{proof}
Informally, this lemma says that there is a natural weak equivalence
$$ \hocolim_k s^{-d_k}\Sinfty X(k)_{d_k} \xra{\sim} \hocolim_k X(k).$$

\begin{cor} \label{hocolim cor}Given a spectrum $X$ and an increasing sequence of integers
$ 0 \leq d_0 < d_1 < d_2 < \dots$,
the homotopy colimit of
\begin{equation*}
\xymatrix @-.8pc{
s^{-d_1}\Sinfty \Sigma^{d_1-d_0}X_{d_0} \ar[d]_{\wr} \ar[dr] & s^{-d_2}\Sinfty \Sigma^{d_2-d_1}X_{d_1} \ar[d]_{\wr} \ar[dr] & s^{-d_3}\Sinfty \Sigma^{d_3-d_2}X_{d_2} \ar[d]_{\wr} \ar[dr] &  \\
s^{-d_0}\Sinfty X_{d_0} & s^{-d_1}\Sinfty X_{d_1} & s^{-d_2}\Sinfty X_{d_2} &  \dots }
\end{equation*}
is naturally weakly equivalent to $X$.
\end{cor}
\begin{proof} Apply the lemma to the case when $X(k) = X$ for all $k$.
\end{proof}
Informally, this corollary
says that there is a natural weak equivalence
$$ \hocolim_k s^{-d_k}\Sinfty X_{d_k} \xra{\sim} X.$$

\subsection{Periodic homotopy}  We recall some of the terminology and big theorems used when one studies homotopy from the chromatic point of view.  A good general reference for this material is Doug Ravenel's book \cite{ravenel}.

We let $\C \subset ho(\Sp)$ denote the stable homotopy category of $p$--local finite CW spectra, and then we let $\C_n \subset \C$ be the full subcategory with objects the $K(n-1)_*$--acylic spectra.  The categories $\C_n$ are properly nested \cite{mitchell}:
$$\C = \C_0 \supset \C_1 \supset \C_2 \supset \dots .$$
An object $F \in \C_n - C_{n+1}$ is said to be of {\em type $n$}.

For finite spectra, the remarkable work of Ethan Devanitz, Mike Hopkins, and Jeff Smith \cite{dhs} tells us the following.

\begin{thm}[Nilpotence Theorem {\cite[Thm.3]{hs}}] Given $F \in \C$, a map $v: \Sigma^d F \ra F$ is nilpotent if and only if $K(n)_*(v)$ is nilpotent for all $n \geq 0$.
\end{thm}

The next two consequences were proved by Hopkins and Smith.

\begin{thm}[Thick Subcategory Theorem \cite{hs,ravenel}]  A nonempty full subcategory of $\C$ that is closed under taking cofibers and retracts is $\C_n$ for some $n$.
\end{thm}

Given $F \in \C$, a map $v: \Sigma^d F \ra F$ is called a {\em $v_n$--self map} if $K(n)_*(v)$ is an isomorphism, while $K(m)_*(v)$ is nilpotent for all $m \neq n$.

\begin{thm}[Periodicity Theorem \cite{hs,ravenel}]
\noindent{\bf (a)} $F \in \C_n$ if and only if $F$ has a $v_n$--self map. \\

\noindent{\bf (b)} Given $F,F^{\prime} \in \C_n$ with $v_n$--self maps $u: \Sigma^c F \ra F$ and $v: \Sigma^dF^{\prime} \ra F^{\prime}$, and $f: F \ra F^{\prime}$, there exist integers $i,j$ such that $ic = jd$ and the diagram
\begin{equation*}
\xymatrix{
\Sigma^{ic} F\ar[d]^{v^i} \ar[r]^{\Sigma^{ic}f} & \Sigma^{jd}F^{\prime} \ar[d]^{v^j}  \\
F \ar[r]^f & F^{\prime} }
\end{equation*}
commutes.
\end{thm}

Given $F \in \C$ of type $n$, we let $T(F)$ denote the mapping telescope of a $v_n$--self map.  An immediate consequence of the Periodicity Theorem is that $T(F)$ is independent of choice of self map.  Furthermore, one deduces that the Bousfield class of $T(F)$ is independent of the choice of type $n$ spectrum $F$.  In other words, If $F$ and $F^{\prime}$ are both of type $n$, then
$$ T(F) \sm Y \simeq * \text{ if and only if } T(F^{\prime}) \sm Y  \simeq *.$$
It is usual to let $T(n)$ ambiguously denote $T(F)$ for any particular type $n$ finite spectrum $F$.

Another consequence of the Periodicity Theorem was proved by the author in \cite{k1}.

\begin{prop}[{\cite[Cor.4.3]{k1}}] \label{resolution prop} There exists a diagram in $\C$,
\begin{equation*}
\xymatrix{
F(1) \ar[d] \ar[r]^{f(1)} & F(2) \ar[dl] \ar[r]^{f(2)} & F(3) \ar[dll] \ar[r] & \dots  \\
S^0 &&& }
\end{equation*}
such that each $F(k) \in \C_n$, and $\displaystyle \hocolim_k F(k) \ra S^0$ induces an $T(m)_*$--isomorphism for all $m \geq n$.
\end{prop}

\begin{rem}  The statement of this proposition deserves some comment, as homotopy colimits of general diagrams in a triangulated category like $ho(\Sp)$ are not always defined.  However, the hocolimit of a sequence as above {\em is} defined (as the cofiber of an appropriate map between coproducts of the $F(k)$).  We note also that only this construction is used in the proof of the proposition given in \cite{k1}; in other words, the proposition is proved working solely in the triangulated homotopy category.
\end{rem}

We give standard names to some localization functors.  Let $L_n^f: \Sp \ra \Sp$ denote localization with respect to $T(0) \vee \dots \vee T(n)$, and then define functors $C_{n-1}^f, M_n^f: \Sp \ra \Sp$ by letting $C_{n-1}^fX$ be the homotopy fiber of $X \ra L_{n-1}^f X$ and $M_n^fX$ be the homotopy fiber of $L_n^fX \ra L_{n-1}^fX$.

These functors are all {\em smashing}, e.g. $X \sm L_n^f S^0 \simeq L_n^f X$ for all $X$, and from this one can quite easily deduce that $L_{T(n)}$ and $M_n^f$ determine each other.  More precisely, there are natural equivalences $L_{T(n)} M_n^f X \simeq L_{T(n)}X$ and $C_{n-1}^f L_{T(n)}X \simeq M_n^f L_{T(n)}X \simeq M_n^f X$

An alternative proof of \propref{resolution prop} occurs in \cite[proof of Thm. 12.1]{bousfield3}, where Bousfield notes that $C_{n-1}^fS^0$ can be written in the form $\displaystyle \hocolim_k F(k)$ with each $F(k) \in \C_n$.  This same result also was proved by Mahowald and Sadofsky in \cite[Proposition 3.8]{mahowaldsadofsky}.

We end this section with characterizations of spectra that are $T(n)$--local or in the image of $M_n^f$.

\begin{lem} \label{Tn local lem} Consider the following three properties that a spectrum $X$ might satisfy. \\

\noindent{(i)} $[F,X] = 0$ for all $F \in \C_{n+1}$. \\

\noindent{(ii)} $[Y,X] = 0$ whenever $F(n) \sm Y \simeq *$ for some type $n$ finite spectrum $F(n)$.  \\

\noindent{(iii)} $T(i)\sm X \simeq *$ for $0 \leq i \leq n-1$. \\

\noindent Properties (i) and (ii) hold if and only if $X$ is $T(n)$--local.  Properties (i) and (iii) hold if and only if $X \simeq M^f_nX$.
\end{lem}
\begin{proof} We can assume that $T(n)$ is the telescope of a $v_n$--self map $v:\Sigma^d F(n) \ra F(n)$.

Condition (i) is equivalent to the statement that $X$ is $L_n^f$--local, while property (ii) says that $X$ is $F(n)$--local.  Thus if $X$ is $T(n)$--local, both (i) and (ii) are true.

If condition (i) holds, so that $X$ is $L_n^f$--local,  we  observe that $v: F(n) \sm X \ra \Sigma^{-d} F(n) \sm X$ is an equivalence, as the cofiber is null, since it can be written (using $S$--duality) in the form $\MapS(F,X)$ with $F$ of type $n+1$.  It follows that $F(n) \sm X \simeq T(n) \sm X$, and thus
$$F(n)_*(X) \simeq T(n)_*(X) \simeq T(n)_*(L_{T(n)}X) \simeq F(n)_*(L_{T(n)}X).$$  Thus if condition (ii) also holds, so that $X$, as well as $L_{T(n)}X$, is $F(n)$--local, we conclude that $X \simeq L_{T(n)}X$, i.e. X is $T(n)$--local.

Finally, property (iii) says that $L_{n-1}^fX \simeq *$, so that $M_n^f X \simeq L_n^f X$.
\end{proof}

\section{Telescopic functors associated to a self map of a space} \label{phi_v section}

\subsection{The basic construction}

Given a space $B$ and a map $v:\Sigma^d B \ra B$ with $d>0$, we define a functor
$$ \Phi_v: \T \ra \Sp$$
as follows.  If $n\equiv -e \mod d$, with $0\leq e \leq d-1$, we let

$$ \Phi_v(Z)_n = \Omega^e \MapT(B,Z).$$
The structure maps $\Phi_v(Z)_{n} \ra \Omega \Phi_v(Z)_{n+1}$ identify with the identity unless $n\equiv 0 \mod d$, in which case it equals the map
$$ v(Z): \MapT(B,Z) \xra{v^*} \MapT(\Sigma^dB,Z) = \Omega^d\MapT(B,Z).$$

The construction is functorial in $v$ in the following sense: a commutative diagram
\begin{equation*}
\xymatrix{
 \Sigma^dA \ar[d]^u \ar[r]^{\Sigma^d f} & \Sigma^dB \ar[d]^v  \\
A \ar[r]^f & B }
\end{equation*}
induces a natural transformation $f^*:\Phi_v(Z) \ra \Phi_u(Z)$.

We list some basic properties of $\Phi_v(Z)$ in the next omnibus lemma.

\begin{lem} \label{big phi lemma}
\noindent{\bf (a)} $\pi_*(\Phi_v(Z)) = v^{-1}\pi_*(Z;B)$.\\

\noindent{\bf (b)} If a map of spaces $Y \ra Z$ induces an isomorphism on $\pi_*$ for $* \gg 0$, then $\Phi_v(Y) \ra \Phi_v(Z)$ is a stable equivalence.  In particular, the $r$--connected covering map $Z\langle r \rangle \ra Z$ induces a stable equivalence
$\Phi_v(Z\langle r \rangle) \ra \Phi_v(Z)$ for all $r$. \\

\noindent{\bf (c)} $v^*: \Phi_v(Z) \ra \Phi_{\Sigma^dv}(Z)$ is a stable equivalence. \\

\noindent{\bf (d)} For all spaces $A$, $\Phi_v(\MapT(A,Z)) = \Phi_{1_A \sm v}(Z) = \MapS(A,\Phi_v(Z))$.  In particular, $\Phi_{\Sigma^c v}(Z) = \Omega^c\Phi_v(Z)$ for all $c$. \\

\noindent{\bf (e)} $\Phi_v$ takes weak equivalences to level weak equivalences (and thus stable equivalences),
fibrations to level fibrations, and homotopy pullbacks to level homotopy pullbacks (and thus stable homotopy pullbacks). \\

\noindent{\bf (f)} Given a commutative diagram
\begin{equation*}
\xymatrix{
 \Sigma^dA \ar[d]^u \ar[r]^{\Sigma^d f} & \Sigma^dB \ar[d]^v \ar[r]^{\Sigma^d g} & \Sigma^dC \ar[d]^w  \\
A \ar[r]^f & B \ar[r]^g & C, }
\end{equation*}
if $A \xra{f} B \xra{g} C$ is a homotopy cofiber sequence of spaces, then the induced sequences
$$ \Phi_w(Z) \xra{g^*} \Phi_v(Z) \xra{f^*} \Phi_u(Z)$$
are homotopy fibration sequences of spectra for all $Z$. \\

\noindent{\bf (g)} If $B$ is a finite CW complex, there is a natural stable equivalence
$$\hocolim_d \Phi_v(Z_d) \simeq \Phi_v(\hocolim_d Z_d)$$
for all diagrams $Z_1 \ra Z_2 \ra Z_3 \ra \dots$ of spectra. \\

\noindent{\bf (h)} If $v_0,v_1: \Sigma^d B \ra B$ are homotopic maps, then $\Phi_{v_0}(Z)$ is naturally stably equivalent to $\Phi_{v_1}(Z)$. \\

\noindent{\bf (i)} There is a natural stable equivalence $\Phi_v(Z) \xra{\sim} \Phi_{v^r}(Z)$, where $v^r: \Sigma^{rd} B \ra B$ denotes the evident $r$--fold composition of $v$ with its various suspensions.
\end{lem}
\begin{proof}  All of this is quite easily verified.  Part (a) is clear, and then parts (b) and (c) follow by check of homotopy groups.  Part (d) follows by inspection, since $\MapT(A \sm B, Z) = \MapT(A,\MapT(B,Z))$.  Parts (e) and (f) follow from the fact that $\MapT(B,Z)$ takes cofibrations in the $B$-variable and fibrations in the $Z$--variable to fibrations.  Similarly, part (g) follows from the fact that $\displaystyle \hocolim_n \MapT(B,Z_n) \simeq \MapT(B, \hocolim_n Z_n)$ if $B$ is a finite complex.

To check part (h), suppose $v_0,v_1: \Sigma^d B \ra B$ are homotopic maps.  If $H: \Sigma^d B \sm I_+ \ra B$ is a homotopy from $v_0$ to $v_1$, let $V: \Sigma^d B \sm I_+ \ra B \sm I_+$ be defined by $V(x,t) = (H(x,t),t)$.  Then there is a commutative diagram
\begin{equation*}
\xymatrix{
\Sigma^d B \ar[d]^{v_0} \ar[r]^-{\Sigma^di_0} & \Sigma^d B\sm I_+ \ar[d]^{V} &  \Sigma^d B \ar[l]_-{\Sigma^di_1} \ar[d]^{v_1}  \\
B \ar[r]^-{i_0} & B \sm I_+ & B, \ar[l]_-{i_1} }
\end{equation*}
which induces natural equivalences
$$ \Phi_{v_0}(Z) \xleftarrow[\sim]{i_0^*} \Phi_{V}(Z) \xrightarrow[\sim]{i_1^*}\Phi_{v_1}(Z).$$

Finally the stable equivalence of part (i) is defined as follows. Write $n$ in the form $n = mrd -sd - e$, with $0\leq s \leq r-1$ and $0 \leq e \leq d-1$.  Then let
$$ \Phi_v(Z)_n \ra \Phi_{v^r}(Z)_n$$
be the map $$\displaystyle \Omega^e\MapT(B,Z) \xra{\Omega^e (v^{s})^*} \Omega^e \MapT(\Sigma^{sd} B, Z) = \Omega^{e+sd}\MapT(B,Z).$$
\end{proof}

\begin{cor} \label{phiv cor} $\Phi_v(Z)$ is a periodic spectrum with period $d$: $\Phi_v(Z) \simeq \Omega^d \Phi_v(Z)$.  Furthermore, the induced functor $\Phi_v: ho(\T) \ra ho(\Sp)$ is determined by the stable homotopy class of $v^r$ for any $r$.
\end{cor}
\begin{proof}  Combine properties (c), (d), (h), and (i) of the lemma.
\end{proof}

Our last property needs some notation. Given an unstable map $u: \Sigma^c A \ra A$ and $X \in \Sp$, let
$u^{-1}\MapS(A,X)$ denote the homotopy colimit of the diagram $$ \MapS(A,X) \xra{u^*} \MapS(\Sigma^{c}A,X) \xra{u^*} \MapS(\Sigma^{2c}A,X)\ra \dots.$$

\begin{lem} \label{smash lemma}
Given maps $u: \Sigma^c A \ra A$ and $v: \Sigma^d B \ra B$, there are natural stable equivalences $u^{-1}\MapS(A, \Phi_v(Z)) \simeq \Phi_{u\sm v}(Z) \simeq v^{-1}\MapS(B,\Phi_u(Z))$.
\end{lem}
\begin{proof}[Sketch proof]  By symmetry, we need just verify the first of these equivalences.
By \lemref{big phi lemma}(d), $u^{-1}\MapS(A, \Phi_v(Z))$ is equal to
$$ \hocolim \{ \Phi_{A \sm v}(Z) \xra{u^*} \Phi_{\Sigma^cA \sm v}(Z) \xra{u^*} \Phi_{\Sigma^{2c}A \sm v}(Z) \xra{u^*} \dots\}.$$
By \lemref{hocolim lemma}, this is stably equivalent to
$$ \hocolim_k s^{-kd}\Sinfty \MapT(\Sigma^{kc}A \sm B, Z).$$
This, in turn, maps to
$$ \hocolim_k s^{-k(c+d)}\Sinfty \MapT(A \sm B, Z),$$
using evident natural maps of the form $\Sinfty \Omega^r W \ra s^{-r} \Sinfty W$, and a check of homotopy groups shows this map between homotopy colimits is an equivalence. Finally, by \lemref{hocolim lemma} again, this last homotopy colimit is equivalent to $\Phi_{u\sm v}(Z)$.
\end{proof}

\subsection{Identifying $\Phi_v(\Oinfty Z)$.}

Recall that, for $X \in \Sp$, $\Oinfty X = X_0$.  The following elementary `swindle' is critical to our arguments.  Note that it says that the functor that assigns $v^{-1}\MapS(B,X)$ to a spectrum $X$ depends only on the space $X_0$.

\begin{prop}[Compare with {\cite[Prop.3.3(4)]{k1}}] \label{phi omega prop}  If $X \in \Sp$ is fibrant (i.e. is an $\Omega$--spectrum), then, given $v: \Sigma^d B \ra B$, there is a natural weak equivalence
$$ \Phi_v(\Oinfty X) \simeq v^{-1}\MapS(B,X).$$
\end{prop}
\begin{proof}  We have natural equivalences:
\begin{equation*}
\begin{split}
v^{-1}\MapS(B,X) &
\xleftarrow{\sim} \hocolim_r s^{-rd} \Sinfty \MapS(\Sigma^{rd}B,X)_{rd} \\
  & = \hocolim_r s^{-rd} \Sinfty \MapT(\Sigma^{rd}B,X_{rd}) \\
  & = \hocolim_r s^{-rd} \Sinfty \MapT(B,\Omega^{rd}X_{rd}) \\
    & \xleftarrow{\sim} \hocolim_r s^{-rd} \Sinfty \MapT(B,X_0) \\
    & \xra{\sim} \Phi_v(\Oinfty X).
\end{split}
\end{equation*}
Here the first equivalence follows from \lemref{hocolim lemma}, the last equivalence similarly follows from \corref{hocolim cor}, and the second to last equivalence holds because $X$ is fibrant.
\end{proof}

\begin{rem}  It is not easy to spot the analogue of this proposition in \cite{bousfield3}, but \cite[Thm.11.9]{bousfield3} is a more elaborate result of this type, and its proof, given in \cite[\S\S  11.10--11.11]{bousfield3} uses arguments very similar to our proof of \lemref{hocolim lemma}.
\end{rem}

\begin{rem} Using the proposition, we can give an alternative proof of part of \lemref{smash lemma}: that $u^{-1}\MapS(A, \Phi_v(Z)) \simeq v^{-1}\MapS(B,\Phi_u(Z))$. If we let $\Phi_u^{fib}(Z)$ be a fibrant replacement for $\Phi_u(Z)$, then $\displaystyle \Oinfty \Phi_u^{fib}(Z) \simeq \hocolim_r \MapT(\Sigma^{rc}A,Z)$.  Thus
\begin{equation*}
\begin{split}
v^{-1}\MapS(B,\Phi_u(Z)) &
\simeq \Phi_v(\hocolim_r \MapT(\Sigma^{rc}A,Z)) \text{ (by the proposition)} \\
  & \simeq \hocolim_r \Phi_v(\MapT(\Sigma^{rc}A,Z)) \\
  & = \hocolim_r \MapS(\Sigma^{rc}A,\Phi_v(Z)) = u^{-1}\MapS(A, \Phi_v(Z)).
\end{split}
\end{equation*}
\end{rem}

\section{$\Phi_v$ when $v$ is a $v_n$--self map of a space} \label{phi v:part 2}

Note that if $v: \Sigma^d B\ra B$ is nilpotent, $\Phi_v(Z)$ will be contractible for all $Z$.  So that this might not be the case, in this section, we study the case when $v$ is a $v_n$--self map of a finite CW complex $B$ of type $n$.

First we discuss a construction in the homotopy category of spectra.

Given $F \in \C_{n}$, it is convenient to let $\Phi(F,Z) \in ho(\Sp)$ denote $\Sigma^t\Phi_u(Z)$, where $u: \Sigma^c A \ra A$ is an unstable $v_n$--map of a finite CW complex $A$ such that $\Sigma^tF \simeq \Sinfty A$.  Similarly, given a map $f: F \ra F^{\prime}$ between finite spectra in $\C_n$, we define $f^*: \Phi(F^{\prime},Z) \ra \Phi(F,Z)$ to be $\Sigma^t\alpha^*: \Sigma^t \Phi_v(Z) \ra \Sigma^t \Phi_u(Z)$ where
\begin{equation*}
\xymatrix{
 \Sigma^dA \ar[d]^u \ar[r]^{\Sigma^d \alpha} & \Sigma^dB \ar[d]^v  \\
A \ar[r]^{\alpha} & B }
\end{equation*}
is a commutative diagram of spaces with $v_n$--self maps, and $\Sigma^tf \simeq \Sinfty \alpha$.

\begin{lem} \label{phi functor lem} $\Phi: \C_n^{op} \times ho(\T) \ra ho(\Sp)$ is a well defined functor and satisfies the next two properties. \\

\noindent{\bf (a)} $\Phi$ takes cofibration sequences in the $\C_n$-variable to fibration sequences in $ho(\Sp)$. \\

\noindent{\bf (b)} $\MapS(F,\Phi(F^{\prime},Z)) \simeq \Phi(F^{\prime} \sm F,Z) \simeq \MapS(F^{\prime},\Phi(F,Z)).$
\end{lem}

\begin{proof}  This follows from the Periodicity Theorem and the results in the last section.
\end{proof}

Now we prove that, when $v$ is a $v_n$--self map, $\Phi_v: \T \ra \Sp$ satisfies versions of the properties listed in \thmref{main thm}.

\begin{thm} \label{v Tn thm}  Let $v: \Sigma^d B\ra B$ is an unstable $v_n$--self map. \\

\noindent{\bf (1)} $\Phi_v(Z) \simeq M_n^f\Phi_v(Z)$ and is also $T(n)$--local, for all spaces $Z$. \\

\noindent{\bf (2)} $\Phi_v(\Oinfty X) \simeq \MapS(B,L_{T(n)} X)$ for all fibrant $X \in \Sp$.
\end{thm}

\begin{proof}[Proof of \thmref{v Tn thm}(1)]
We need to verify that $\Phi_v(Z)$ satisfies the three properties listed in \lemref{Tn local lem}.

Property (i) says that $[F, \Phi_v(Z)] = 0$ for all $F \in \C_{n+1}$.  To see this, we first note that, since $\Phi_v(Z)$ is periodic, we can assume that $F = \Sinfty A$ for some finite CW complex $A$ of type at least $n+1$.  But then $[F, \Phi_v(Z)] = \pi_0(\MapS(A,\Phi_v(Z)) = 0$, because
$$ \MapS(A,\Phi_v(Z)) = \Phi_{1_A \sm v}(Z) \simeq *,$$
as $A \sm B$ will have type greater than $n$, so that $1_A \sm v: \Sigma^d A \sm B \ra A \sm B$ will be nilpotent.

Property (ii) says that, with $F(n)$ a fixed finite spectrum of type $n$, $[Y,\Phi_v(Z)] = 0$ whenever $F(n) \sm Y \simeq *$.  To prove this, we make use of the properties of the functor $\Phi$ listed in \lemref{phi functor lem}.

So suppose that $F(n) \sm Y \simeq *$.  Let $$\C_Y = \{ F \in \C_n \ | \ [Y,\Phi(F,Z)]_* = 0 \text{ for all } Z\}.$$
Using the Thick Subcategory Theorem, we check that $\C_Y = \C_n$, thus verifying property (ii). Firstly, $\C_Y$ is a thick subcategory by \lemref{phi functor lem}(a).  Secondly, it contains at least one type $n$ complex, as it contains all type $n$ complexes of the form $F(n) \sm F$, with $F$ of type $n$.  To see this, using \lemref{phi functor lem}(b), we have
$$ [Y,\Phi(F(n) \sm F,Z)]_* = [Y, \MapS(F(n),\Phi(F,Z)]_* = [Y \sm F(n),\Phi(F,Z)]_* = 0.$$

Property (iii) says that $T(i) \sm \Phi_v(Z) \simeq *$ for $i \leq n-1$.  We can assume that $T(i)$ is the telescope of the $S$--dual of an unstable $v(i)$--map $u: \Sigma^c A \ra A$, where $A$ is a finite CW complex of type $i$.  Then
\begin{equation*}
\begin{split}
T(i) \sm \Phi_v(Z) &  \simeq u^{-1}\MapS(A,\Phi_v(Z)) \\
  & \simeq v^{-1}\MapS(B,\Phi_u(Z)) \text{ (by \lemref{smash lemma})} \\
  & = \hocolim_r \Omega^{rd} \Phi_{u\sm 1_B}(Z)) \\
  & \simeq *,
\end{split}
\end{equation*}
as $A \sm B$ has type greater than $i$, so that $u \sm 1_B: \Sigma^c A \sm B \ra A \sm B$ is nilpotent, and thus $\Phi_{u\sm 1_B}(Z) \simeq *$.
\end{proof}

\begin{proof}[Proof of \thmref{v Tn thm}(2)]  This is similar to the author's proof of \cite[Prop. 3.4]{k1}. Thanks to \propref{phi omega prop}, we just need to show that, if $v: \Sigma^d B \ra B$ is a $v_n$--map, then there is a weak equivalence $$v^{-1}\MapS(B,X) \simeq \MapS(B,L_{T(n)}X).$$

This is easy to do.  We claim that each of the maps
$$ v^{-1}\MapS(B,X) \ra v^{-1}\MapS(B,L_{T(n)}X) \la \MapS(B,L_{T(n)}X)$$
 is an equivalence.  If we let $T(n)$ be modeled by the telescope of the dual of $v$, then the first map identifies with the equivalence $T(n) \sm X \xra{\sim} T(n) \sm L_{T(n)}X$. The second map is an equivalence as $v$ is a $T(n)_*$--isomorphism, so that each map in the diagram $$\MapS(B,L_{T(n)}) \xra{v^*} \MapS(\Sigma^d B, L_{T(n)})\xra{v^*} \MapS(\Sigma^{2d} B, L_{T(n)}) \ra \dots$$ is an equivalence.
 \end{proof}

\section{The construction of $\Phi_n$ and the proof of \thmref{main thm}} \label{phi_n section}

\subsection{The construction on the level of homotopy categories}

Recall that we have a functor
$$\Phi: \C_n^{op} \times ho(\T) \ra ho(\Sp)$$
defined by letting $\Phi(F,Z)$ denote $\Omega^t\Phi_u(Z)$, where $u: \Sigma^c A \ra A$ is an unstable $v_n$--map of a finite CW complex $A$ such that $\Sigma^tF \simeq \Sinfty A$.

Now consider a resolution of $S^0$ as in \propref{resolution prop}: a diagram
\begin{equation} \label{ho resolution}
\xymatrix{
F(1) \ar[d]^(.38){q(1)} \ar[r]^{f(1)} & F(2) \ar[dl]^(.2){q(2)} \ar[r]^{f(2)} & F(3) \ar[dll]^(.2){q(3)} \ar[r] & \dots  \\
S^0 &&& }
\end{equation}
such that each $F(k) \in \C_n$, and such that the map
$$\displaystyle q = \hocolim_k q(k):\hocolim_k F(k) \ra S^0$$
induces an isomorphism in $T(m)_*$ for all $m \geq n$.

\begin{defn}  Define $\Phi_n^T: ho(\T) \ra ho(\Sp)$ by the formula
$$ \Phi^T_n(Z) = \holim_k \Phi(F(k),Z)$$
\end{defn}

We have the following theorem, which is \thmref{main thm} on the level of homotopy categories.

\begin{thm} \label{ho thm} $\Phi_n^T$ satisfies the following properties. \\

\noindent{\bf (1)} $\Phi_n^T(Z)$ is $T(n)_*$--local for all $Z \in ho(\T)$. \\

\noindent{\bf (2)} $\MapS(F,\Phi_n^T(Z)) \simeq \Phi(F,Z) \in ho(\Sp)$ for all $F \in \C$ and  $Z\in ho(\T)$. \\

\noindent{\bf (3)} \ $\Phi_n^T(\Oinfty X) \simeq L_{T(n)}X$, for all $X \in ho(\Sp)$.
\end{thm}

\begin{proof}  By \thmref{v Tn thm}(1), each $\Phi(F(k),Z)$ is $T(n)_*$--local.  Since the homotopy limit of local objects is again local, statement (1) follows.

To see that (2) is true, given $F \in \C$, we compute in $ho(\Sp)$:
\begin{equation*}
\begin{split}
\MapS(F,\Phi_n^T(Z)) &
\simeq  \MapS(F, \holim_k \Phi(F(k),Z))\\
  & \simeq  \holim_k \MapS(F, \Phi(F(k),Z)) \\
  & \simeq  \holim_k \MapS(F(k), \Phi(F,Z)) \\
  & \simeq  \MapS(\hocolim_k F(k), \Phi(F,Z)) \\
  & \xla{\sim} \MapS(S^0, \Phi(F,Z)) = \Phi(F,Z).
\end{split}
\end{equation*}
Here the third equivalence is an application of \lemref{phi functor lem}(b), while the last map is an equivalence because it is induced by the $T(n)_*$--isomorphism
$q$ and $\Phi(F,Z)$ is $T(n)_*$--local (by \thmref{v Tn thm}(1)).

The proof that (3) is true is similar:
\begin{equation*}
\begin{split}
\Phi_n^T(\Oinfty X) &
= \holim_k \Phi(F(k),\Oinfty X) \\
&\simeq \holim_k \MapS(F(k),L_{T(n)}X) \text{ (by \thmref{v Tn thm}(2))} \\
& = \MapS(\hocolim_k F(k), L_{T(n)}X) \\
  & \xla{\sim} \MapS(S, L_{T(n)}X) = L_{T(n)}X.
\end{split}
\end{equation*}
\end{proof}

\subsection{Rigidifying the construction}

\begin{defn} A {\em rigidification} of diagram (\ref{ho resolution}) consists of the following data. \\

\noindent{(i)}  Finite complexes $B(k)$ of type $n$. \\

\noindent{(ii)}  Natural numbers $d(k)$ such that $d(k)|d(k+1)$ together with unstable $v_n$--self maps $v(k): \Sigma^{d(k)}B(k) \ra B(k)$. \\

\noindent{(iii)} Natural numbers $t(k)$ such that $t(k) \leq t(k+1)$ together with
maps $p(k): B(k) \ra S^{t(k)}$ and $\beta(k): \Sigma^{e(k)}B(k) \ra B(k+1)$, where $e(k) = t(k+1)-t(k)$. \\

These are required to satisfy three properties: \\

\noindent{(a)} $\Sinfty B(k) \in \Sp$ represents $\Sigma^{t(k)}F(k) \in ho(\Sp)$, $\Sinfty p(k)$ represents $\Sigma^{t(k)} q(k)$, and $\Sinfty \beta(k)$ represents $\Sigma^{t(k+1)}f(k)$.\\

\noindent{(b)} With $r(k) = d(k+1)/d(k)$, the diagram
\begin{equation*}
\xymatrix{
\Sigma^{e(k) + d(k+1)}B(k) \ar[d]^{\Sigma^{e(k)}v(k)^{r(k)}} \ar[rr]^-{\Sigma^{d(k+1)}\beta(k)} && \Sigma^{d(k+1)}B(k+1) \ar[d]^{v(k+1)}  \\
\Sigma^{e(k)}B(k) \ar[rr]^-{\beta(k)} && B(k+1) }
\end{equation*}
commutes in $\T$. \\

\noindent{(c)} The diagram
\begin{equation*}
\xymatrix{
\Sigma^{e(k)}B(k) \ar[rr]^-{\beta(k)} \ar[dr]^{\Sigma^{e(k)}p(k)} && B(k+1) \ar[dl]_{p(k+1)}  \\
& S^{t(k+1)} &  }
\end{equation*}
commutes.
\end{defn}

\begin{lem} Rigidifications exist.
\end{lem}

\begin{proof}[Sketch proof]  This is basically Bousfield's construction of `an admissible spectral $L_n^f$--cospectrum' given in \cite[Thm.12.1]{bousfield3}.  One proceeds by induction on $k$.  Having constructed $B(k)$, $v(k)$, and $p(k)$, using the Periodicity Theorem in the stable range, one chooses $e(k)$ so large that there exists $B(k+1)$, $v(k+1)$, $p(k+1))$, and $\beta(k)$ making property (a) hold, and so that the diagrams in (b) and (c) commute up to homotopy. Then one replaces $B(k+1)$ and $\beta(k)$, so that the new $\beta(k)$ is a cofibration.  Finally, one uses the homotopy extension property of cofibrations (applied to both $\beta(k)$ and $\Sigma^{d(k+1)}\beta(k)$) to replace $v(k+1)$ and $p(k+1)$ by  homotopic maps so that the new diagrams (b) and (c) strictly commute.
\end{proof}

Given a rigidification of (\ref{ho resolution}), we will make use of two families of induced natural maps.

The maps $\beta(k):\Sigma^{e(k)}B(k) \ra B(k+1)$ induce a natural maps $$ \Phi_{v(k+1)}(Z) \ra \Phi_{\Sigma^{e(k)}v(k)^{r(k)}}(Z)= \Omega^{e(k)}\Phi_{v(k)^{r(k)}}(Z).$$
Adjointing these, and suspending $t(k)$--times, yield  natural maps
$$ \beta(k)^*: \Sigma^{t(k+1)}\Phi_{v(k+1)}(Z) \ra  \Sigma^{t(k)}\Phi_{v(k)^{r(k)}}(Z).$$

The `top cell' maps $p(k): B(k) \ra S^{t(k)}$ induce maps
$$ s^{-t(k)}X \ra \Omega^{t(k)}X = \MapS(S^{t(k)}, X) \ra \MapS(B(k),X).$$
Adjointing these, yield natural maps
$$ p(k)^*: X \ra s^{t(k)} \MapS(B(k),X).$$

The last maps are compatible as $k$ varies, and so induce a natural map
$$ p^*: X \ra \holim_k s^{t(k)} \MapS(B(k),X).$$

\begin{lem} \label{little lem} $p^*$ is an equivalence if $X$ is $T(n)$--local.
\end{lem}
\begin{proof}  In the homotopy category, $p(k)^*$ corresponds to $$q(k)^*: X \ra \MapS(F(k),X)$$ so that $p^*$ corresponds to
$$ q^*: X = \MapS(S^0,X) \ra \MapS(\hocolim_k F(k),X).$$
This is an equivalence if $X$ is $T(n)$--local.
\end{proof}

\begin{defn} \label{Phi defn} Given a rigidification of (\ref{ho resolution}), we define $\Phi_n: \T \ra \Sp$ by letting $\Phi_n(Z)$ be the homotopy limit of the diagram
\begin{equation*}
\xymatrix @-.8pc{
& \Sigma^{t(3)}\Phi_{v(3)^{r(3)}}(Z) & \Sigma^{t(2)}\Phi_{v(2)^{r(2)}}(Z) & \Sigma^{t(1)}\Phi_{v(1)^{r(1)}}(Z) \\
\dots  \ar[ur] & \Sigma^{t(3)}\Phi_{v(3)}(Z) \ar[u]_{\wr} \ar[ur]_-{\beta(2)^*} & \Sigma^{t(2)}\Phi_{v(2)}(Z) \ar[u]_{\wr} \ar[ur]_-{\beta(1)^*} & \Sigma^{t(1)}\Phi_{v(1)}(Z). \ar[u]_{\wr}
 }
\end{equation*}
\end{defn}

Informally, we write $\displaystyle \Phi_n(Z) = \holim_k \Sigma^{t(k)} \Phi_{v(k)}(Z)$.

\begin{proof}[Proof of \thmref{main thm}]

By construction, in $ho(\Sp)$, $\Phi_n(Z)$  represents the holimit of the diagram
$$ \dots \ra \Phi(F(3),Z) \xra{f(2)^*} \Phi(F(2),Z) \xra{f(1)^*} \Phi(F(1),Z),$$
i.e. $\Phi_n^T(Z)$.  The various properties of $\Phi_n$ stated in \thmref{main thm} are verified by giving proofs similar to those given in proving the analogous properties of $\Phi^T_n$ listed in \thmref{ho thm}, with the constructions of natural equivalences `rigidifying' as needed in straightforward ways.

We run through some details.

Property (1) is clear: $\Phi_n(Z)$ is the homotopy limit of $T(n)$--local spectra, thus is itself $T(n)$--local.

For property (2), we have
\begin{equation*}
\begin{split}
\MapS(B,\Phi_n(Z)) &
=  \MapS(F, \holim_k \Sigma^{t(k)}\Phi_{v(k)}(Z))\\
  & =  \holim_k \MapS(B, \Sigma^{t(k)}\Phi_{v(k)}(Z)) \\
  & \xla{\sim}  \holim_k \Sigma^{t(k)} \MapS(B, \Phi_{v(k)}(Z)) \\
  & \simeq \holim_k \Sigma^{t(k)} \MapS(B(k), \Phi_{v}(Z)) \text{ (by \lemref{smash lemma})}\\
  & \xra{\sim} \holim_k s^{t(k)}\MapS(B(k), \Phi_{v}(Z)) \\
  & \xla{\sim} \Phi_v(Z),
\end{split}
\end{equation*}
since $\Phi_{v}(Z)$ is $T(n)$--local.

For property (3), we have
\begin{equation*}
\begin{split}
\Phi_n(\Oinfty X) &
= \holim_k \Sigma^{t(k)}\Phi_{v(k)}\Oinfty X) \\
& \xra{\sim} \holim_k s^{t(k)}\Phi_{v(k)}\Oinfty X) \\
&\simeq \holim_k s^{t(k)}\MapS(B(k),L_{T(n)}X) \text{ (by \thmref{v Tn thm}(2))} \\
  & \xla{\sim} L_{T(n)}X \text{ (by \lemref{little lem})}.
\end{split}
\end{equation*}

Finally suppose that a functor $\Phi^{\prime}_n: \T \ra \Sp$ satisfies the next two properties, analogues of properties (1) and (2). \\

\noindent{($1^{\prime}$)} \ $\Phi^{\prime}_n(Z)$ is $T(n)$--local, for all spaces $Z$.  \\

\noindent{($2^{\prime}$)} \ There is a weak equivalence of spectra $\Map_{\Sp}(B,\Phi^{\prime}_n(Z)) \simeq \Phi_v(Z)$, for all unstable $v_n$ self maps $v:\Sigma^d B \ra B$, natural in both $Z$ and $v$. \\

\noindent Then we have:
\begin{equation*}
\begin{split}
\Phi^{\prime}_n(Z) & \xrightarrow{\sim} \holim_k s^{t(k)}\MapS(B(k), \Phi^{\prime}_n(Z)) \text{ (by ($1^{\prime}$))}\\
& \xleftarrow{\sim} \holim_k \Sigma^{t(k)}\MapS(B(k), \Phi^{\prime}_n(Z)) \\
  & \simeq \holim_k \Sigma^{t(k)} \Phi_{v(k)}(Z)) \text{ (by ($2^{\prime}$))}\\
  & = \Phi_n(Z).
\end{split}
\end{equation*}
Thus properties (1) and (2) characterize $\Phi_n$.
\end{proof}

\section{Bousfield's adjoint $\Theta_n$} \label{theta_n section}

In \cite{bousfield3}, Bousfield constructs a functor $\Theta_n: \Sp \ra \T$, which serves as a left adjoint of sorts to $\Phi_n$.  In this section, we run through how this works.

\subsection{The construction of $\Theta_v$.}

Given a self map of a space $v: \Sigma^d B \ra B$, the functor
$$ \Phi_v: \T \ra \Sp$$
admits a left adjoint
$$ \Theta_v: \Sp \ra \T,$$
defined as follows.  Given a spectrum $X$ with iterated structure maps $\sigma_r: \Sigma^d X_{rd} \ra X_{(r+1)d}$, $\Theta_v(X)$ is defined to be the coequalizer of the two maps
$$ \bigvee_r \Sigma^d B \sm X_{rd} \begin{array}{c}  v \\[-.08in] \longrightarrow \\[-.1in] \longrightarrow \\[-.1in] \sigma
\end{array} \bigvee_r B \sm X_{rd},$$
where, on the $rd^{th}$ wedge summand, $v$ is $\Sigma^d B \sm X_{rd} \xra{v \sm 1} B \sm X_{rd}$, while $\sigma$ is $ \Sigma^d B \sm X_{rd} \simeq B \sm \Sigma^d X_{rd} \xra{1 \sm \sigma_r} B \sm X_{(r+1)d}$.

It is easy and formal to check that $\Theta_v$ and $\Phi_v$ form an adjoint pair.  However, to be homotopically meaningful, one would like these functors to form a Quillen pair, so that they induce an adjunction on the associated homotopy categories.  For this to be true, it is necessary and sufficient to check that $\Phi_v$ preserves trivial fibrations, and also fibrations between fibrant objects. (See, e.g. \cite[Lem.10.5]{bousfield3} or \cite[Prop.8.5.4]{hirschhorn}.)

$\Phi_v$ certainly preserves trivial fibrations, as it takes a trivial fibration to a levelwise trivial fibration, which will then be a stable trivial fibration.

Suppose that $W \ra Z$ is a fibration in $\T$. In the stable model category structure, $\Phi_v(W) \ra \Phi_v(Z)$ will be a fibration between fibrant objects only if the obvious necessary condition holds: $\Phi_v(W)$ and $\Phi_v(Z)$ must both be fibrant, i.e. $\Omega$--spectra.

Unravelling the definitions, $\Phi_v(Z)$ will be an $\Omega$--spectrum if and only if the map
$$ v^*: \Map(B,Z) \ra \Map(\Sigma^dB,Z)$$
is a weak equivalence.

One can {\em force} this condition to be true as follows.  Let $L_v: \T \ra \T$ denote localization with respect to the map $f$, and then let $\T_v$ denote $\T$ with the associated model category structure in which weak equivalences $\T_v$ are maps $f$ so that $L_vf$ is a weak equivalence in $\T$.  (See e.g. \cite{hirschhorn} for these constructions and many references to the literature.)  We recover a variant of \cite[Lem.10.6]{bousfield3}.

\begin{lem} \label{adjoint lemma} For any $v: \Sigma^d B \ra B$, we have the following. \\

\noindent{\bf (a)} $\Theta_v: \Sp \ra \T_v$ and $\Phi_v: \T_v \ra \Sp$ form a Quillen pair. \\

\noindent{\bf (b)} For all $X \in \Sp$ and $Z \in \T$, there is a natural bijection
$$ [\Theta_v(X), L_vZ] \simeq [X, \Phi_v(L_vZ)].$$
\end{lem}

\subsection{Periodic localization of spaces.}

In \cite[\S 4.3]{bousfield3}, Bousfield defines
$$ L_n^f: \T \ra \T$$
to be localization with respect to the map $\Sigma A \ra *$, where $A$ is chosen so that $\Sinfty A$ is equivalent to a finite spectrum of type $n+1$, and the connectivity of $H_*(A;\Z/p)$ is chosen to be as low as possible.  His proof that this is independent of choice appears in \cite[Thm.9.15]{bousfieldJAMS}, and depends on the Thick Subcategory Theorem.

For our purposes, $L_n^f: \T \ra \T$ satisfies two elementary properties that we care about.

\begin{lem} \label{local Oinfy lemma}If $X$ is a $L_n^f$--local spectrum, then $\Oinfty X$ is $L_n^f$--local space.
\end{lem}

\begin{proof} Let $A$ be chosen as in the definition of $L_n^f: \T \ra \T$. For all $t$, $\pi_t(\MapT(\Sigma A, \Oinfty X)) = [\Sinfty \Sigma^{t+1}A,X] = 0$, since $L_n^f$--local spectra admit no nontrivial maps from objects in $\C_{n+1}$. Thus $\MapT(\Sigma A, \Oinfty X) \simeq *$, and so $\Oinfty X$ is $L_n^f$--local.
\end{proof}

\begin{lem} \label{double susp lemma} If $Z$ is a $L_n^f$--local space, then it is also $L_v$--local for all unstable $v_n$--self maps $v: \Sigma^d B \ra B$ that are double suspensions.
\end{lem}

\begin{proof} Since $v$ is a double suspension, it fits into a cofibration sequence of the form
$$ \Sigma C \ra \Sigma^d B \xra{v} B \ra \Sigma^2 C,$$ where $\Sinfty C$ has type $n+1$.  This induces a fibration sequence
$$ \MapT(\Sigma^2 C, Z) \ra \MapT(B, Z) \xra{v^*} \MapT(\Sigma^d B, Z) \ra \MapT(\Sigma C, Z),$$
in which the first and last of these mapping spaces are null if $Z$ is $L_n^f$--local.  Thus the middle map is an equivalences, and so $Z$ is $L_v$--local.
\end{proof}

A deeper property of $L_n^f$ goes as follows.

\begin{prop} \label{Ln prop} If $v$ is a $v_n$--self map, the natural map $\Phi_v(Z) \ra \Phi_v(L_n^fZ)$ is a stable equivalence.
\end{prop}
\begin{proof} We wish to show that the map of spaces
$$ \hocolim_r \MapT(\Sigma^{rd}B,Z) \ra \hocolim_r \MapT(\Sigma^{rd}B,L_n^fZ)$$
induces an isomorphism on homotopy groups (in high dimensions).  This is pretty much \cite[Theorem 11.5]{bousfieldJAMS}, and we sketch how the proof goes.

The map we care about factors in the homotopy category:
\begin{equation*} 
\xymatrix{
 \hocolim_r \MapT(\Sigma^{rd}B,Z) \ar[d] \ar[r] & \hocolim_r \MapT(\Sigma^{rd}B,L_n^fZ)  \\
L_n^f \hocolim_r \MapT(\Sigma^{rd}B,Z)  & \hocolim_r L_n^f \MapT(\Sigma^{rd}B,Z), \ar[l]_{\sim} \ar[u]}
\end{equation*}
where the indicated equivalence is \cite[Lemma 11.6]{bousfieldJAMS}.

The right vertical arrow induces an isomorphism on homotopy groups in high dimensions, due to \cite[Theorem 8.3]{bousfieldJAMS}, a general result which describes to what extent functors like $L_{\Sigma C}$ preserve fibrations.  Applied to the case in hand, one learns that there is a number $\delta$ such that the natural map
$$ L_n^f \MapT(C,Z) \ra \MapT(C, L_n^f Z)$$
will induce an isomorphism on $\pi_i$ for $i \geq \delta$ for all finite complexes $C$. Thus our right vertical map will induce isomorphisms on $\pi_i$ in the same range.

It follows that we need just check that the left vertical map is an equivalence, or, otherwise said, that $\displaystyle \hocolim_r \MapT(\Sigma^{rd}B,Z)$ is $L_n^f$--local.  Recalling that $L_n^f = L_{\Sigma A}$ for a well chosen $A$ of type $(n+1)$, we have
\begin{equation*}
\begin{split}
\MapT(\Sigma A, \hocolim_r \MapT(\Sigma^{rd}B,Z)) &
\simeq \hocolim_r \MapT(\Sigma A, \MapT(\Sigma^{rd}B,Z) \\
  & = \hocolim_r \MapT(\Sigma A \sm \Sigma^{rd}B,Z) \\
  & \simeq *,
\end{split}
\end{equation*}
since $1_{\Sigma A} \sm v$ will be nilpotent by the Nilpotence Theorem.
\end{proof}

\begin{rem} It would interesting to have a proof of this proposition that avoided the use of \cite[Theorem 8.3]{bousfieldJAMS}.
\end{rem}

Combining this proposition with \lemref{adjoint lemma} and \lemref{double susp lemma} yields the next theorem.

\begin{thm} \label{theta v thm} If $v: \Sigma^d B \ra B$ is a $v_n$--self map and a double suspension, there is a natural bijection
$$ [\Theta_v(X), L_n^fZ] \simeq [X,\Phi_v(Z)],$$
for all $Z \in \T$ and $X \in \Sp$.
\end{thm}

\subsection{The definition of $\Theta_n$.}

Let the following data make up a rigidification of diagram (\ref{ho resolution}), as used in the definition of $\Phi_n$: \\

\noindent{(i)}  Finite complexes $B(k)$ of type $n$. \\

\noindent{(ii)}  Natural numbers $d(k)$ such that $d(k)|d(k+1)$ together with unstable $v_n$--self maps $v(k): \Sigma^{d(k)}B(k) \ra B(k)$. \\

\noindent{(iii)} Natural numbers $t(k)$ such that $t(k) \leq t(k+1)$ together with
maps $p(k): B(k) \ra S^{t(k)}$ and $\beta(k): \Sigma^{e(k)}B(k) \ra B(k+1)$, where $e(k) = t(k+1)-t(k)$. \\

By double suspending everything, we can also assume that each $v(k)$ is a double suspension.

\begin{defn}  Given this data, we define $\Theta_n: \Sp \ra \T$ by letting $\Theta_n(Z)$ be the homotopy colimit of the diagram
\begin{equation*}
\xymatrix @-.8pc{
\Theta_{v(1)^{r(1)}}(\Omega^{t(1)}X) \ar[d] \ar[dr]^-{\beta(1)_*}& \Theta_{v(2)^{r(2)}}(\Omega^{t(2)}X) \ar[d] \ar[dr]^-{\beta(2)_*}& \Theta_{v(3)^{r(3)}}(\Omega^{t(3)}X) \ar[d] \ar[dr]^-{\beta(3)_*}& \\
\Theta_{v(1)}(\Omega^{t(1)}X)& \Theta_{v(2)}(\Omega^{t(2)}X)& \Theta_{v(3)}(\Omega^{t(3)}X) & \dots, \\
 }
\end{equation*}
where each vertical map will be an $L_n^f$--equivalence, and each $\beta(k)_*$ is itself a natural zig-zag diagram
$$ \Theta_{v(k)^{r(k)}}(\Omega^{t(k)}X) \xla{\sim} \Theta_{v(k)^{r(k)}}(\Sigma^{e(k)}\Omega^{t(k+1)}X) \xra{\beta(k)_*} \Theta_{v(k+1)}(\Omega^{t(k)}X).$$
\end{defn}

Informally, we write $\displaystyle \Theta_n(X) = \hocolim_k \Theta_{v(k)}(\Omega^{t(k)}X)$.

From \thmref{theta v thm}, we deduce
\begin{thm}[{\cite[Theorem 5.4(iii)]{bousfield3}}] \label{adjoint thm} There is a natural bijection
$$ [\Theta_n(X), L_n^fZ] = [X, \Phi_n(Z)]$$
for all $Z \in \T$ and $X \in \Sp$.
\end{thm}

\begin{proof}  The idea is that, since $\Phi_n$ is the limit of functors of the form $\Phi_v$, and $\Theta_n$ is the colimit of their adjoints $\Theta_v$, the theorem should follow from \thmref{theta v thm}.  The only detail needing a careful check is that the zig-zag natural map $$ \Theta_{v(k)^{r(k)}}(\Omega^{t(k)}X) \xla{\sim} \Theta_{v(k)^{r(k)}}(\Sigma^{e(k)}\Omega^{t(k+1)}X) \xra{\beta(k)_*} \Theta_{v(k+1)}(\Omega^{t(k+1)}X)$$
used in the definition of $\Theta_n$ above really is adjoint to the more directly defined map
$$ \Sigma^{t(k+1)}\Phi_{v(k+1)}(Z) \xra{\beta(k)^*}  \Sigma^{t(k)}\Phi_{v(k)^{r(k)}}(Z)$$
used in the definition of $\Phi_n$.

To see this we have a commutative diagram:
{\tiny
\begin{equation*}
\xymatrix @-1.2pc{
\MapT(\Theta_{v(k+1)}\Omega^{t(k+1)}X,Z) \ar[dd]^{\wr} \ar[r]^-{\beta(k)^*} & \MapT(\Theta_{v(k)^{r(k)}}\Sigma^{e(k)}\Omega^{t(k+1)}X,Z) \ar[dd]^{\wr} &  \MapT(\Theta_{v(k)^{r(k)}}\Omega^{t(k)}X,Z)\ar[dd]^{\wr} \ar[l]_-{\sim} \\
&& \\
\MapT(\Theta_{v(k+1)}s^{-t(k+1)}X,Z) \ar@{=}[dd] \ar[r]^-{\beta(k)^*} & \MapT(\Theta_{v(k)^{r(k)}}\Sigma^{e(k)}s^{-t(k+1)}X,Z) \ar@{=}[dd] &  \MapT(\Theta_{v(k)^{r(k)}}s^{-t(k)}X,Z) \ar@{=}[dd] \ar[l]_-{\sim} \\
&& \\
\MapS(X,s^{t(k+1)}\Phi_{v(k+1)}Z)  \ar[r]^-{\beta(k)^*}       & \MapS(X,s^{t(k+1)}\Omega^{e(k)}\Phi_{v(k)^{r(k)}}Z)       &  \MapS(X,s^{t(k)}\Phi_{v(k)^{r(k)}}Z)      \ar[l]_-{\sim} \\
&& \\
\MapS(X,\Sigma^{t(k+1)}\Phi_{v(k+1)}Z) \ar[uu]_{\wr} \ar[r]^-{\beta(k)^*} & \MapS(X,\Sigma^{t(k)}\Phi_{v(k)^{r(k)}}Z) \ar[uu]_{\wr} &  \MapS(X,\Sigma^{t(k)}\Phi_{v(k)^{r(k)}}Z). \ar[uu]_{\wr} \ar@{=}[l] \\
}
\end{equation*}
}

\end{proof}

\begin{rem} The lack of elegance in the proof of the `detail' checked above reflects the fact, though our functor $\Phi_n: \T \ra \Sp$ is intuitively the homotopy limit of right adjoint functors $\Phi_{v(k)}$, it does {\em not} make up the right part of an adjoint pair.  Note that our official definition involves the use of $\Sigma: \Sp \ra \Sp$, which induces an equivalence of homotopy categories, but is a left adjoint, not a right one. It doesn't seem possible to somehow replace $\Sigma^{t(k)}$ by $s^{t(k)}$ (which {\em is} a right adjoint) in Definition \ref{Phi defn}.  This same problem shows up in Bousfield's construction: see the paragraph before \cite[Theorem 11.7]{bousfield3}.
\end{rem}

\begin{cor} In $ho(\Sp)$, there is a natural equivalence
$$ L_{T(n)}\Sinfty \Theta_n(X) \simeq L_{T(n)}X.$$
\end{cor}
\begin{proof}  In the last theorem, let $Z$ be the space $\Oinfty L_{T(n)}Y$, which is $L_n^f$--local  by \lemref{local Oinfy lemma}. We see that, for all $X,Y \in \Sp$, there are natural isomorphisms
\begin{equation*}
\begin{split}
[\Sinfty \Theta_n(X), L_{T(n)}Y] &
\simeq [\Theta_n(X), \Oinfty L_{T(n)}Y] \\
  & \simeq [X, \Phi_n(\Oinfty L_{T(n)}Y)] \\
  & \simeq [X,L_{T(n)}Y].
\end{split}
\end{equation*}
The corollary then follows from Yoneda's lemma.
\end{proof}

\begin{rem} The careful reader will note that this corollary is {\em not} dependent on \propref{Ln prop} (and thus not dependent on Bousfield's careful study of the behavior of localized fibration sequences), as we have derived it by only applying our other results to a space $Z$ that is $L_n^f$--local.
\end{rem}

\begin{rem}  From the corollary, it follows that the Telescope Conjecture is equivalent to the statement that if a {\em space} is $K(n)_*$--acyclic, then it is $T(n)_*$--acyclic.
\end{rem}

\section{The section $\eta_n$} \label{eta_n section}

One of the main applications of the functor $\Phi_n$, is that it leads to the construction of a natural transformation
$$\eta_n(X): L_{T(n)}X \ra L_{T(n)} \Sinfty \Oinfty X$$
which is a natural homotopy section of the $T(n)$--localization of the evaluation map
$$ \epsilon(X): \Sinfty \Oinfty X \ra X.$$
The construction is immediate: $\eta_n$ is defined by applying $\Phi_n$ to the natural map
$$\eta(\Oinfty X): \Oinfty X \ra Q \Oinfty X.$$

This section is both used and studied in \cite{k3,rezk}.

It seems plausible that $\eta_n$ is the {\em unique} section of $L_{T(n)}\epsilon$.  We have a couple of partial results along these lines.

The first was discussed in \cite{k4}.

\begin{prop} $\eta_n$ is unique up to `tower phantom' behavior in the Goodwillie tower for $\Sinfty \Oinfty$ in the following sense: for all $d$, the composite
$$ L_{T(n)}X \xra{\eta_n(X)} L_{T(n)} \Sinfty \Oinfty X \xra{L_{T(n)}e_d(X)} L_{T(n)}P_d^{\infty}(X)$$
is the unique natural section of $L_{T(n)}P_d^{\infty}(X) \xra{L_{T(n)}p_d(X)} L_{T(n)}X$. \end{prop}
Here $\Sinfty \Oinfty X \xra{e_d(X)} P_d^{\infty}(X)$ is the $d^{th}$ stage of the Goodwillie tower, and $p_d$ is the canonical natural transformation such that $\epsilon = p_d \circ e_d$. The uniqueness asserted in the proposition is an immediate consequence of the main theorem of \cite{k2}.

Our second observation is in the spirit of observations by Rezk in \cite{rezk}.  As usual, we let $QZ$ denote $\Oinfty \Sinfty Z$.

\begin{prop} Any natural transformation $f(X): X \ra L_{T(n)}\Sinfty \Oinfty X$ will be determined by $f(S^0): S^0 \ra L_{T(n)}\Sinfty QS^0$.
\end{prop}
\begin{proof}  We begin by showing that we can reduce to the case when $X = \Sinfty Z$, a suspension spectrum.  As the range of $f$ is $T(n)$--local, we can extend the domain to $L_{T(n)}X$.  In the diagram
\begin{equation*}
\xymatrix{
L_{T(n)}\Sinfty \Oinfty X \ar[d]^{\epsilon_*} \ar[rr]^-{f(\Sinfty \Oinfty X)} && L_{T(n)}\Sinfty Q\Oinfty X \ar[d]^{\epsilon_*}  \\
L_{T(n)}X \ar[rr]^-{f(X)} && L_{T(n)}\Sinfty \Oinfty X, }
\end{equation*}
the left vertical map has a section given by $f(\eta_n(X))$. Thus the bottom map is determined by the top map.

Next we observe that any continuous functor $G: \T \ra \Sp$ comes with a natural transformation $Z \sm G(W) \ra G(Z \sm W)$, and this structure is natural in $G$.  Applied to our situation,  for any space $Z$, we have a commutative diagram
\begin{equation*}
\xymatrix{
Z \sm \Sinfty S^0 \ar[d]^{\wr} \ar[rr]^-{1_Z \sm f(S^0)} && Z \sm L_{T(n)}\Sinfty QS^0 \ar[d]  \\
\Sinfty Z \ar[rr]^-{f(Z)} && L_{T(n)}\Sinfty QZ. }
\end{equation*}
Thus the bottom map is determined by the top.
\end{proof}

\begin{quest} Is it true that the map
$$ L_{T(n)}\Sinfty QS^0 \ra \prod_d L_{T(n)}\Sinfty B\Sigma_{d+},$$
arising from the James--Hopf maps $QS^0 \ra QB\Sigma_{d+}$,  is monic on $\pi_0$?
\end{quest}

If so, then the last propositions combine to show that $\eta_n$ is the unique natural section to the localized evaluation map.

\section{A guide to computations} \label{computations}

In this section we briefly survey calculations that have been made of $\Phi_n(Z)$ for various pairs $(n,Z)$.  Before jumping into this, we first explain that there is also interest in explicit calculations of $\pi_*(\Phi_n(Z))$, or variants thereof.

\subsection{Periodic homotopy groups of spaces}

Recall that there is a sequence of spectra
$$ F(1) \ra F(2) \ra F(3) \ra \dots $$
such that each $F(k)$ is finite of type $n$ and $\displaystyle \hocolim_k F(k) \simeq C_{n-1}^fS^0$.

Dualizing this in the stable homotopy category, one gets a diagram
$$ \D F(1) \la  \D F(2) \la \D F(3) \la \dots. $$
Furthermore, each $\D F(k)$ comes with a $v_n$--self map which we will generically call `$v$', these are compatible in the usual way, and any given finite part of this data `eventually' desuspends to spaces.  Thus the following definition makes sense.

\begin{defn} The $v_n$--periodic homotopy groups of a space $Z$ are defined to be $$v_n^{-1}\pi_*(Z) = \colim_k v^{-1}\pi_*(Z;\D F(k)).$$
\end{defn}

\begin{ex} When $n=1$, one takes $F(k)$ to be a Moore spectrum of type $Z/p^k$, and the self maps are `Adams maps'.  Using traditional notation, $$v_1^{-1}\pi_*(Z) = \colim_k v^{-1}\pi_{*+1}(Z;\Z/p^k).$$
\end{ex}

\begin{lem} The groups $v_n^{-1}\pi_*(Z)$ can be rewritten in terms of $\Phi_n(Z)$ in various ways:
\begin{equation*}
\begin{split}
v_n^{-1}\pi_*(Z) & = \colim_k \pi_*(\MapS(\D F(k),\Phi_n(Z))) \\
  & = \colim_k \pi_*(F(k) \sm \Phi_n(Z)) \\
  & = \pi_*(C_{n-1}^f\Phi_n(Z)) \\
  & = \pi_*(M_{n}^f\Phi_n(Z)).
\end{split}
\end{equation*}
\end{lem}

The next lemma relates $\Phi_n$--equivalences to isomorphisms on localized homotopy groups.

\begin{lem} \label{v-equiv lemma} Given a map $f: W \ra Z$ between spaces, the following conditions are equivalent. \\

\noindent{\bf (a)} $\Phi_n(f): \Phi_n(W) \ra \Phi_n(Z)$ is a weak equivalence. \\

\noindent{\bf (b)} $f_*: v_n^{-1}\pi_*(W) \ra v_n^{-1}\pi_*(Z)$ is an isomorphism. \\

\noindent{\bf (c)} $f_*: v^{-1}\pi_*(W;B) \ra v^{-1}\pi_*(Z;B)$ is an isomorphism for all unstable $v_n$--self maps $v: \Sigma^d B \ra B$. \\

\noindent{\bf (d)} $f_*: v^{-1}\pi_*(W;B) \ra v^{-1}\pi_*(Z;B)$ is an isomorphism for some unstable $v_n$--self map $v: \Sigma^d B \ra B$.
\end{lem}
\begin{proof} Let $g = \Phi_n(f)$.  Then condition (a) says that $g$ is an equivalence, (b) says that $M_n^fg$ is an equivalence, and conditions (c) and (d) say that $\D B \sm g$ is an equivalence for appropriate $B$'s.  As $g$ is a map between $T(n)$--local spectra, the three conditions are all equivalent: clearly (a) implies all the other statements, $L_{T(n)}M_n^f g \simeq g$ so that (b) implies (a), (c) obviously implies (d), and finally (d) implies that $v^{-1}\D B \sm g$ is an equivalence, so that $g$ is a $T(n)_*$--isomorphism and thus (a) holds.
\end{proof}

\subsection{Basic observations}

From properties of $\Phi_v$ listed in \lemref{big phi lemma}, one deduces the next two useful basic calculational rules.

\begin{lem} $\MapS(A,\Phi_n(Z)) \simeq \Phi_n(\MapT(A,Z))$ for all $A,Z \in \T$.
\end{lem}

\begin{lem} $\Phi_n$ takes homotopy pullbacks in $\T$ to homotopy pullbacks in $\Sp$.
\end{lem}

One might wonder to what extent $\Phi_n$ might take the homotopy limit (`microscope') of a sequence $ Z_1 \la Z_2 \la Z_3 \la \dots$ to the corresponding holimit in $\Sp$.  Unfortunately this will not always be the case; the correct statement can be formally deduced from \thmref{adjoint thm}.

\begin{lem} Given a sequence of spaces $ Z_1 \la Z_2 \la Z_3 \la \dots$, we have
$$ \holim_k \Phi_n(Z_k) \simeq \Phi_n(\holim_k L_n^f Z_k).$$
\end{lem}

Since $\Phi_n(\holim_k Z_k) \simeq \Phi_n(L_n^f \holim_k Z_k)$, one sees that the failure of $\Phi_n$ to commute with microscopes is caused by the failure of $L_n^f: \T \ra \T$ to commute with microscopes.

More constructively, one has the following consequence of \lemref{v-equiv lemma}.

\begin{lem} \label{convergence lemma} Given a sequence of spaces $ Z_1 \la Z_2 \la Z_3 \la \dots$, the natural map
$$\Phi_n(\holim_k Z_k) \ra \holim_k \Phi_n(Z_k)$$
is an equivalence if and only if
$$ v^{-1}\lim_k \pi_*(Z_k;B) \ra \lim_k v^{-1}\pi_*(Z_k;B)$$
is an isomorphism for some unstable $v_n$--self map $v: \Sigma^d B \ra B$.
\end{lem}

Since $\Phi_n(Z)$ can be `calculated' as $L_{T(n)}X$ if $Z = \Oinfty X$, the following strategy for computing $\Phi_n(Z)$ emerges:  try to `resolve' $Z$ by towers of fibrations with fibers which are infinite loopspaces, and hope that \lemref{convergence lemma} can be applied when needed.

\subsection{$\Phi_n(S^m)$ when $m$ is odd}

The strategy just described was implimented in
beautiful work of Arone and Mahowald \cite{aronemahowald} on the Goodwillie tower of the identity functor. It allows for the identification of a short resolution of $\Phi_n(Z)$ with `known' composition factors, when $Z$ is an odd dimensional sphere.  We will be brief here; for a slightly different overview of how this goes, see the last sections of our survey paper \cite{k4}.

We need some notation. Let $m \rho_k$ denote the direct sum of $m$ copies of the reduced real regular representation of $V_k = (\mathbb Z/p)^k$.  Then $GL_k(\mathbb Z/p)$ acts on the Thom space $BV_k^{m\rho_k}$.  Let $e_k \in \mathbb Z_{(p)}[GL_k(\mathbb Z/p)]$ be any idempotent in the group ring representing the Steinberg module, and then let $L(k,m)$ be the associated stable summand of  $BV_k^{m\rho_k}$:
$$ L(k,m) = e_k \Sinfty BV_k^{m\rho_k}.$$
The spectra $L(k,0)$ and $L(k,1)$ agree with spectra called $M(k)$ and $L(k)$ in the literature from the early 1980's: see e.g. \cite{mitchellpriddy,kuhnpriddy}.  Two properties of the $L(k,m)$ play a crucial roles for our purposes:
\begin{itemize}
\item When $m$ is odd, the cohomology  $H^*(L(k,m); \mathbb Z/p)$ is free over the finite subalgebra $\A(k-1)$ of the Steenrod algebra $\A$.
\item Fixing $m$, the connectivity of $L(k,m)$ has a growth rate like $p^k$.
\end{itemize}
The first fact here implies that $L(k,m)$ is $T(n)_*$--acylic for $k>n$.  Indeed, the $E_2$--term of the Adams spectral sequence which computes $[B,L(k,m)]_*$ for any finite $B$ will have a vanishing line of small enough slope so that one can immediately deduce that $v^{-1}E_2 = 0$ if $v$ is a $v_n$--self map of $B$.

Arone and Mahowald's analysis in \cite{aronemahowald}, supported by \cite{aronedwyer}, shows that, for odd $m$, there is a tower of fibrations under the $p$--local sphere $S^m$:
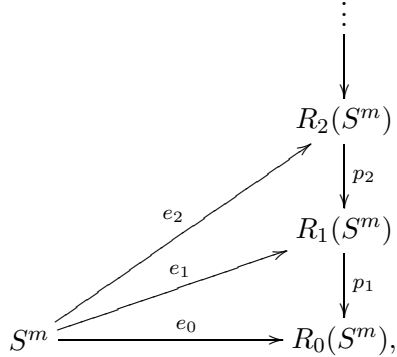
\begin{equation*}
\xymatrix{
&&& \vdots \ar[d] \\
&&& R_2(S^m) \ar[d]^{p_2} \\
&&& R_1(S^m) \ar[d]^{p_1} \\
S^m \ar[rrr]^{e_0}  \ar[urrr]^{e_1} \ar[uurrr]^{e_2} &&& R_0(S^m),
}
\end{equation*}
such that $\displaystyle S^m \simeq \holim_k R_k(S^m)$, $R_0(S^m) = QS^m$, and, for $k \geq 1$, the fiber of $p_k$ is equivalent to $\Oinfty \Sigma^{m-k}L(k,m)$.  (The space $R_k(S^m)$ is the $p^k$th stage of the Goodwillie tower of the identity functor applied to $S^m$.)

Using the two properties of the $L(k,m)$ bulleted above, Arone and Mahowald then deduce that
$$ v^{-1}\lim_k \pi_*(R_k(S^m);B) \ra \lim_k v^{-1}\pi_*(R_k(S^m);B)$$
is an isomorphism for any self map $v: \Sigma^d B \ra B$. See \cite[\S 4.1]{aronemahowald}. It follows that
$$ e_{n*}: v^{-1}\pi_*(S^m;B) \ra v^{-1}\pi_*(R_n(S^m);B)$$
is an isomorphism for any $v_n$--self map $v: \Sigma^d B \ra B$.  One deduces the following about $\Phi_n(S^m)$.

\begin{thm}[see {\cite[Theorem 7.18]{k4}}] \label{phi theorem}  Let $m$ be odd.  The map
$$ \Phi_n(e_n): \Phi_n(S^m) \ra \Phi_n(R_n(S^m))$$
is an equivalence.  Thus the spectrum $\Phi_n(S^m)$ admits a finite decreasing filtration with fibers $L_{T(n)}\Sigma^{m-k}L(k,m)$ for $k = 0, \dots, n$.
\end{thm}

With a little diagram chasing, one can do better than this.  Let  $L(k)_1^{m-1}$ be the fiber of the natural map of spectra $L(k,1) \ra L(k,m)$.  The fibration sequence
of spectra $$ L(k)_1^{m-1} \ra L(k,1) \ra L(k,m)$$
induces a short exact sequence in mod $p$ cohomology, and is thus split as $\A(k-1)$--modules.  By applying $\Phi_n$ to the fiber sequence $\Omega_0^m S^m \ra S^1 \ra \Omega^{m-1}S^m$ and applying the previous theorem, one deduces an improved result.

\begin{thm}[{\cite[Theorem 7.20]{k4}}] \label{phi theorem 2}  Let $m$ be odd.  The spectrum $\Phi_n(S^m)$ admits a finite decreasing filtration with fibers $L_{T(n)}\Sigma^{m+1-k}L(k)_1^{m-1}$ for $k = 1, \dots, n$.
\end{thm}

\begin{ex} When $p=2$, $L(1)_1^m = \mathbb RP^{m}$.  Specializing to $n=1$,  we learn that, for there is a weak equivalence
$$ \Phi_1(S^{2k+1}) \simeq L_{T(1)}\Sigma^{2k+1}\mathbb RP^{2k}.$$
Specializing to $n=2$, we learn that there is a fibration sequence of spectra
$$ \Phi_2(S^{2k+1}) \ra L_{T(2)}\Sigma^{2k+1}\mathbb RP^{2k} \ra L_{T(2)}\Sigma^{2k} L(2)_1^{2k}.$$
The first of these is equivalent to an older theorem of Mahowald \cite{mahowald} that said that the James--Hopf map $\Omega^{2k} S^{2k+1} \ra Q \Sigma \mathbb RP^{2k}$ induces an isomorphism on $v_1$--periodic homotopy groups.  The odd prime version of this is due to Rob Thompson \cite{thompson}.
\end{ex}

\subsection{$\Phi_1(Z)$ for many $Z$}

There is a huge amount known about $v_1^{-1}\pi_*(Z)$ thanks to the prodigious efforts of Bousfield, together with Don Davis and his collaborators.  A survey article by Davis \cite{davis} describes computations known by the mid 1990's. In recent years, beginning with \cite{bousfield99}, there has been an explosion of new, more elegantly organized, computations, often explictly describing $\Phi_1(Z)$ enroute: see the references below for entries into the recent literature.

Ingredients special to the $n=1$ case that enter the story include the following.
\begin{itemize}
\item The identification of $\Phi_1(S^{2k+1})$ as described above. 
\item The fact that $L_{T(1)} = L_{K(1)}$.
\item The identification of $L_{K(1)}S^0$ as the fiber of an appropriate map of the form $\Psi^r - 1: KO \ra KO$ \cite{old bousfield}.
\item A tight relationship between a maps of spaces being $K(1)_*$--isomorphisms and being a $\Phi_1$--equivalences \cite{bousfieldJAMS}.
\end{itemize}

In summary, for appropriate spaces $Z$, $v_1^{-1}\pi_*(Z)$ is essentially determined by $KO^*(Z;\Z_p)$, together with Adams operations.  Bousfield's recent careful study \cite{bousfield5} is state of the art in this area.  Davis \cite{davis02} gives many complete calculations when $Z$ is a compact Lie group, with calculations beginning with knowledge of the Lie group's represenation ring.  A very recent amusing result in this spirit is due to Martin Bendersky and Davis \cite{benderskydavis}, and says that there is a 2--primary homotopy equivalence
$$ \Phi_1(DI(4)) \simeq L_{K(1)} \Sigma^{725019}T \sm M,$$
where $DI(4)$ is the Dywer--Wilkerson exotic $2$--compact group, $T$ is the three cell finite spectrum $S^0 \cup_{\eta} e^2 \cup_2 e^3$, and $M$ is a mod $2^{21}$ Moore spectrum.

\end{document}